\documentclass[a4paper,10pt]{amsart}
\usepackage[utf8x]{inputenc}
\usepackage{amssymb}
\usepackage{amsmath}
\usepackage{amsthm}
\usepackage{mathrsfs}
\usepackage{bbm}
\usepackage[all]{xy}

\usepackage{IEEEtrantools}

\usepackage{stmaryrd}

\usepackage{hhline}

\usepackage{tikz}
\usetikzlibrary{matrix,arrows,decorations.pathmorphing}
\usetikzlibrary{positioning}
\usepackage[parfill]{parskip}
\usepackage[active]{srcltx}

\title[$2$-representations of some $2$-categories of projective functors]{Simple 
transitive $2$-representations of some $2$-categories of projective functors}
\author{Jakob Zimmermann}
\date{\today}

\newtheorem{theorem}{Theorem}[section]

\newtheorem{lemma}[theorem]{Lemma}
\newtheorem{corollary}[theorem]{Corollary}
\theoremstyle{definition}
\newtheorem{example}[theorem]{Example}

\newtheorem{remark}[theorem]{Remark}

\usepackage{enumerate}

\begin{document}

\begin{abstract}
We show that every simple transitive $2$-representation of the $2$-category of projective functors 
for a certain quotient of the quadratic dual of the preprojective algebra associated with a 
tree is equivalent to a cell $2$-rep\-re\-sen\-ta\-tion.
\end{abstract}

\maketitle
\section{Introduction}

In \cite{MM1}, Mazorchuk and Miemietz started a systematic study of $2$-representations for 
certain $2$-categories which should be thought of as analogues of finite dimensional algebras. 
They introduced the notion of cell $2$-representations as a possible $2$-analogue 
of the notion of simple modules. This was revised in \cite{MM5,MM6} where the notion of a 
simple transitive $2$-representation was introduced. A weak version of the Jordan-H{\"o}lder theory
was developed in \cite{MM5} for simple transitive $2$-representations which was a convincing argument
that simple transitive $2$-representations are proper $2$-analogues of simple modules. In many
important cases, for example for the $2$-category of Soergel bimodules in type $A$, it turns out that 
every simple transitive $2$-representations is equivalent to a cell $2$-representation.

Another class of natural $2$-categories for which every simple transitive $2$-rep\-re\-sen\-ta\-tions is 
equivalent to a cell $2$-representation is the class of $2$-categories of projective bimodules
for a finite dimensional self-injective associative algebra, see \cite{MM5,MM6}. After \cite{MM5,MM6}
there were several attempts to extend this results to other associative algebras. Two particular
algebras were considered in \cite{MZ1} and one more in \cite{MMZ}. These two papers have rather
different approaches: the approach of \cite{MZ1} is based on existence of a non-zero projective-injective 
module while \cite{MMZ} treats the smallest algebra which does not have any projective-injective modules.
Recently, \cite{MZ2} extended the approach of \cite{MZ1} and completely covered the case of directed algebras 
which have a non-zero projective-injective module. We refer the reader to \cite{Ma} for a general overview 
of the problem and related results.

In this note we show that the method developed in \cite{MZ2} can also be extended to some interesting
algebras which are not directed (but which have a non-zero projective-injective module). The algebras
we consider are certain quotients of quadratic duals of preprojective algebras associated with
trees (cf. \cite{Ri}). These kinds of algebras appear naturally in Lie theory (see \cite{S,M}), in diagram algebras
(see \cite{HK}) and in the theory of Koszul algebras (see \cite{D}). Our main result is that,
for our algebras (which are defined in Subsection~\ref{s2.1}), every simple transitive $2$-representation of the
corresponding $2$-category of projective bimodules is equivalent to a cell $2$-representation.

The paper is organized as follows. In the next section we define the type of algebras which we want 
to study, describe some motivating examples and give all the necessary notions needed to formulate 
the main result. Section~\ref{mainRes} is then devoted to stating and proving the main result.

\section{Preliminaries}

Throughout the paper we work over an algebraically closed field $\Bbbk$.

\subsection{The algebra $A_{T,S}$}\label{s2.1}
Let $n$ be a positive integer. Let $T = (V,E)$ be a tree with vertices labeled
by numbers $1, 2, \ldots, n$, where $n>1$.
We denote by $L \subseteq V$ the set of all leaves of $T$. 
Denote by $Q = Q_T = (V, \hat{E})$ the quiver were we replace every (unoriented)
edge $\{i,j\} \in E$,  by two arrows (i.e. oriented edges) 
$(i, j)$ and $(j, i)$. Let $\Bbbk Q$ be the path algebra of $Q$.

Now we define a certain quotient $A_{T,S}$ of $\Bbbk Q$. For this, fix a
(possibly empty) subset $S$ of $L$, and consider the ideal $\mathcal{I}$ of $\Bbbk Q$ generated
by the following relations:
\begin{itemize}
 \item For all pairwise distinct $v_1,v_2,v_3 \in V$ such that there are arrows $a_1, a_2 \in \hat{E}$ with
 \[
  \begin{tikzpicture}
  \matrix(m)[matrix of math nodes, row sep=2.5em, column sep=2.5em,
    text height=1.0ex, text depth=0.25ex]
    {v_1 & v_2 & v_3, \\};
    \path[->,font=\scriptsize]
    (m-1-1) edge node[above] {$a_1$} (m-1-2);
    \path[->,font=\scriptsize]
    (m-1-2) edge node[above] {$a_2$} (m-1-3);
  \end{tikzpicture}
\]
 we set $a_2a_1 = 0$.
 \item For all pairwise distinct vertices $v_1,v_2,v_3 \in V$ 
 such that there exist arrows $a_1, a_2, b_1, b_2 \in \hat{E}$ with
 \[
  \begin{tikzpicture}
  \matrix(m)[matrix of math nodes, row sep=2.5em, column sep=2.5em,
    text height=1.0ex, text depth=0.25ex]
    {v_1 & v_2 & v_3, \\};
    \path[->,font=\scriptsize]
    (m-1-1) edge [bend left] node[above] {$a_1$} (m-1-2);
    \path[->,font=\scriptsize]
    (m-1-2) edge [bend left] node[below] {$b_1$} (m-1-1);
    \path[->,font=\scriptsize]
    (m-1-2) edge [bend left] node[above] {$a_2$} (m-1-3);
    \path[->,font=\scriptsize]
    (m-1-3) edge [bend left] node[below] {$b_2$} (m-1-2);
  \end{tikzpicture}
\]
 we set $a_1b_1 = b_2a_2$.
 \item For $v \in V$ and $s \in S$ such that there are arrows $a, b \in \hat{E}$ with
  \[
  \begin{tikzpicture}
  \matrix(m)[matrix of math nodes, row sep=2.5em, column sep=2.5em,
    text height=1.0ex, text depth=0.25ex]
    {v & s,  \\};
    \path[->,font=\scriptsize]
    (m-1-1) edge [bend left] node[above] {$a$} (m-1-2);
    \path[->,font=\scriptsize]
    (m-1-2) edge [bend left] node[below] {$b$} (m-1-1);
  \end{tikzpicture}
 \]
 we set $ab = 0$.
\end{itemize} 
The algebra $A_{T,S}$, which we will denote simply by  $A$, is now defined
as the quotient of $\Bbbk Q$ by the ideal $\mathcal{I}$.
We denote the idempotents of $A$ by
$e_i$, for each $i \in V$. For $i\in V$, we set $P_i:=Ae_i$ and denote by 
$L_i$ the simple top of $P_i$.

The structure of projective $A$-modules follows directly from the defining relations:
\begin{itemize}
\item If $i\in V\setminus S$, then $P_i$ is projective-injective of Loewy
length three with isomorphic top and socle. The module $\mathrm{Rad}(P_i)/\mathrm{Soc}(P_i)$
is multiplicity-free and contains all simple $L_j$ such that $\{i,j\}\in E$.
\item If $i\in S=V$, then $n=2$ and $P_i$ is projective-injective of Loewy
length two with non-isomorphic top and socle. 
\item If $i\in S\neq V$, then $P_i$ is not injective, it has Loewy
length two and its socle is isomorphic to $L_j$, where $j\in V$ is the unique vertex such that $\{i,j\}\in E$.
\end{itemize}
From the above description we see that the algebra $A$ is self-injective if and only if $S= \varnothing$
or $S=V$ (in the latter case we have $n=2$).

The motivation for the above definition stems from the following examples. 
\begin{example}
 Let $T$ be the following Dynkin diagram of type $A_n$:
 \[
  \begin{tikzpicture}
  \matrix(m)[matrix of math nodes, row sep=2.5em, column sep=2.5em,
    text height=1.0ex, text depth=0.25ex]
    {1 & 2 & \cdots & n - 1 & n. \\};
    \path[-,font=\scriptsize]
    (m-1-1) edge (m-1-2);
    \path[-,font=\scriptsize]
    (m-1-2) edge (m-1-3);
    \path[-,font=\scriptsize]
    (m-1-3) edge (m-1-4);
    \path[-,font=\scriptsize]
    (m-1-4) edge (m-1-5);
  \end{tikzpicture}
 \]
We have $L = \{1, n\}$. Set $S := \{n\}$. Then $Q_T$ is the following quiver:
 
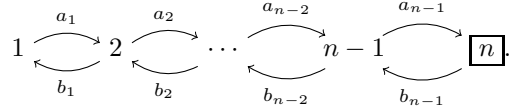
\[
  \begin{tikzpicture}
  \matrix(m)[matrix of math nodes, row sep=2.5em, column sep=2.5em,
    text height=1.0ex, text depth=0.25ex]
    {1 & 2 & \cdots & n - 1 & \boxed{n}. \\};
    \path[->,font=\scriptsize]
    (m-1-1) edge [bend left] node[above] {$a_1$} (m-1-2);
    \path[->,font=\scriptsize]
    (m-1-2) edge [bend left] node[below] {$b_1$} (m-1-1);
    \path[->,font=\scriptsize]
    (m-1-2) edge [bend left] node[above] {$a_2$} (m-1-3);
    \path[->,font=\scriptsize]
    (m-1-3) edge [bend left] node[below] {$b_2$} (m-1-2);
    \path[->,font=\scriptsize]
    (m-1-3) edge [bend left] node[above] {$a_{n-2}$} (m-1-4);
    \path[->,font=\scriptsize]
    (m-1-4) edge [bend left] node[below] {$b_{n-2}$} (m-1-3);
    \path[->,font=\scriptsize]
    (m-1-4) edge [bend left] node[above] {$a_{n-1}$} (m-1-5);
    \path[->,font=\scriptsize]
    (m-1-5) edge [bend left] node[below] {$b_{n-1}$} (m-1-4);
  \end{tikzpicture}
 \]
 The distinguished leaf $n$ is the one which is in $S$.
 The relations in $A = A_{T,S}$ are given by 
 \begin{IEEEeqnarray*}{r;c;l;r}
  a_{i+1}a_i & = & b_{i+1}b_i =  0, & \quad i = 1, \ldots, n-2; \\ 
  a_ib_i & = & b_{i+1}a_{i+1}, & i = 1, \ldots, n-2;\\
  a_{n-1}b_{n-1} & = & 0. &
 \end{IEEEeqnarray*}
 The module category over this algebra is equivalent to the principal block of 
 parabolic category $\mathcal{O}$ associated to the complex Lie algebra $\mathfrak{sl}_n$ 
 and a parabolic subalgebra of $\mathfrak{sl}_n$  for which the semi-simple part of the
 Levi quotient is isomorphic to $\mathfrak{sl}_{n-1}$, see e.g.  \cite{S}.
\end{example}
A second example is:
\begin{example}
 Let $T$ be the following tree on $4$ vertices 
  
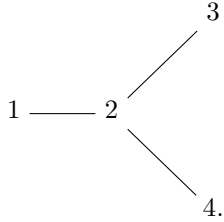
\[
  \begin{tikzpicture}
  \matrix(m)[matrix of math nodes, row sep=2.5em, column sep=2.5em,
    text height=1.0ex, text depth=0.25ex]
    { & & 3\\
    1 & 2 &\\
     &  &  4.\\};
    \path[-,font=\scriptsize]
    (m-2-1) edge (m-2-2);   
    \path[-,font=\scriptsize]
    (m-2-2) edge (m-1-3);
    \path[-,font=\scriptsize]
    (m-2-2) edge (m-3-3);
  \end{tikzpicture}
 \]
 We have $L = \{1, 3, 4\}$. Set $S := \{3, 4\}$. Then $Q_T$ looks as follows:
   
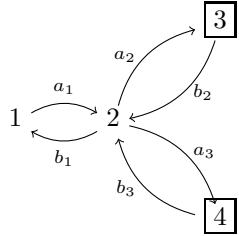
\[
  \begin{tikzpicture}
  \matrix(m)[matrix of math nodes, row sep=2.5em, column sep=2.5em,
    text height=1.0ex, text depth=0.25ex]
    { & & \boxed{3}\\
    1 & 2 &\\
     &  &  \boxed{4}\\};
    \path[->,font=\scriptsize]
    (m-2-1) edge [bend left] node[above] {$a_1$} (m-2-2) 
    (m-2-2) edge [bend left] node[below] {$b_1$} (m-2-1);
    \path[->,font=\scriptsize]
    (m-2-2) edge [bend left] node[left] {$a_2$} (m-1-3)
    (m-1-3) edge [bend left] node[right] {$b_2$} (m-2-2);
    \path[->,font=\scriptsize]
    (m-2-2) edge [bend left] node[right] {$a_3$} (m-3-3)
    (m-3-3) edge [bend left] node[left] {$b_3$} (m-2-2);
  \end{tikzpicture}
 \]
 The relations in $A_{T,S}$ are given by 
 \begin{align*}
  a_2a_1 & = a_3a_1 = b_1b_2 = b_1b_3 = a_3b_2 = a_2b_3 = a_2b_2 =a_3b_3 = 0, \\
  a_1b_1 & = b_2a_2 = b_3a_3.
 \end{align*}
 These kinds of quivers appear as parts of infinite quivers in e.g. \cite{M}.
\end{example}

\subsection{The $2$-category $\mathscr{C}_A$}
From now on we fix a tree $T$ and a subset $S$ of its leaves. Let $A=A_{T,S}$.
For generalities on finitary $2$-categories, we refer the reader to \cite{MM1}.

Following \cite[Subsection 7.3]{MM1}, we define the finitary $2$-category $\mathscr{C}_A$ of 
projective endofunctors of $A$-mod. Fix a small category $\mathcal{C}$ equivalent to $A$-mod. 
The $2$-category $\mathscr{C}_A$ has one object $\texttt{i}$, which we identify with $\mathcal{C}$. 
Indecomposable  $1$-mor\-phisms are endofunctors of $\mathcal{C}$ given by tensoring with:
\begin{itemize}
\item the regular $A$-$A$-bimodule $_AA_A$ (this corresponds to the identity 
$1$-morphism $\mathbbm{1}_{\texttt{i}}$);
\item the indecomposable $A$-$A$-bimodule $Ae_i \otimes_{\Bbbk} e_jA$, for some $i,j\in\{1,2,\dots,n\}$,
we denote such a $1$-morphism by  $F_{ij}$. 
\end{itemize}
Lastly, $2$-morphisms are  homomorphisms of $A$-$A$-bimodules.

A \emph{finitary $2$-representation} of $\mathscr{C}_A$  is a (strict) $2$-functor   from 
$\mathscr{C}_A$ to the $2$-category of small finitary additive $\Bbbk$-linear categories. 
In other words, $\mathbf{M}$ is given by an additive and $\Bbbk$-linear functorial actions 
on a category $\mathbf{M}(\texttt{i})$ which is equivalent to the category $B$-proj of 
projective modules over some finite dimensional associative $\Bbbk$-algebra $B$. 
For $M \in \mathbf{M}(\texttt{i})$, we will often write $FM$ instead of $\mathbf{M}(F)(M)$.
All {finitary $2$-representation} of $\mathscr{C}_A$ form a $2$-category $\mathscr{C}_A$-afmod
where $1$-morphisms are strong $2$-natural transformations and $2$-morphisms are modifications,
see \cite[Section~2.3]{MM3} for details.

We call a finitary $2$-representation $\mathbf{M}$ of $\mathscr{C}_A$ \emph{transitive} 
if, for every non-zero object $X \in \mathbf{M}(\texttt{i})$,
the additive closure of $\{FX\}$, where $F$ runs through all $1$-morphisms in $\mathscr{C}_A$, 
equals $\mathbf{M}(\texttt{i})$.  We call a transitive $2$-representation $\mathbf{M}$ 
\emph{simple} if $\mathbf{M}(\texttt{i})$ has no proper  $\mathscr{C}_A$-invariant ideals. 
For more details on this, we refer the reader to \cite{MM5, MM6}.

One class of examples of simple transitive $2$-representations are so-called cell $2$-representations. 
For details about these we refer the reader to \cite{MM1, MM2}. The $2$-category $\mathscr{C}_A$ has
two two-sided cells: the first one consisting of $\mathbbm{1}_{\texttt{i}}$ and the second one containing
all $F_{ij}$. The first two-sided cell is a left cell. The second two-sided cell contains $n$
different left cells, namely, $\mathcal{L}_j:=\{F_{ij}:i=1,2,\dots,n\}$, where $j=1,2,\dots,n$.
Up to equivalence, we have two cell $2$-representations:
\begin{itemize}
\item The cell $2$-representations $\mathbf{C}_{\mathbbm{1}_{\texttt{i}}}$ which is given as the 
quotient of the left regular action of $\mathscr{C}_A$ on $\mathscr{C}_A(\texttt{i},\texttt{i})$ 
by the unique maximal $\mathscr{C}_A$-invariant ideal (cf. \cite[Section~6]{MM2}). 
\item Each cell $2$-representations $\mathbf{C}_{\mathcal{L}_j}$, where $j=1,2,\dots,n$, 
is equivalent to the defining $2$-rep\-re\-sen\-ta\-tion (the defining action on $\mathcal{C}$).
\end{itemize}

\subsection{Positive idempotent matrices}
One of the ingredients in our proofs is the following classification of non-negative idempotent matrices, see \cite{F69}.
\begin{theorem}
\label{Flor}
 Let $I$ be a non-negative idempotent matrix of rank $k$. Then there exists a permutation
 matrix $P$ such that
 \[P^{-1}IP = 
   \begin{pmatrix}
    0 & AJ & AJB\\
    0 & J & JB \\
    0 & 0 & 0 \\
   \end{pmatrix}
\quad\text{ with }\quad
 J =  \begin{pmatrix}
       J_1 &  0   & \cdots & 0 \\
       0   & J_2 &  &  \vdots\\
       \vdots & & \ddots & 0\\
       0 & \cdots & 0 & J_k
      \end{pmatrix}.
\]
Here, each $J_i$ is a non-negative idempotent matrix of rank one and $A, B$ are non-negative matrices
of the appropriate size.
\end{theorem}

\begin{remark}
This theorem can be applied to quasi-idempotent (but not nilpotent) matrices as well. 
If $I^2 = \lambda I$ and $\lambda\neq 0$, then $(\frac{1}{\lambda}I)^2 = \frac{1}{\lambda^2}I^2 = 
\frac{1}{\lambda}I$. Hence $\frac{1}{\lambda}I$ is an idempotent and thus can be described by the above theorem. 
\end{remark}

\section{Main result}
\label{mainRes}
Fix a tree $T$ and a subset $S$ of its leaves and set $A = A_{T,S}$. Then our main result can be stated as follows:

\begin{theorem}\label{mainresult}
 Let $\mathbf{M}$ be a simple transitive $2$-representation of $\mathscr{C}_A$, then $\mathbf{M}$ is equivalent
 to a cell $2$-representation.
\end{theorem}

Note that, if $S = \varnothing$ or $S=V$, then the algebra $A$ is self-injective and hence 
$\mathscr{C}_A$ is a weakly fiat $2$-category. In this case the statement follows from
\cite[Theorem~15]{MM5} and \cite[Theorem~33]{MM6}. Therefore, in what follows, we assume that 
$S \neq \varnothing,V$.

\subsection{Some notation}
For a simple transitive $2$-representation $\mathbf{M}$ of $\mathscr{C}_A$, 
we denote by $B$ a basic $\Bbbk$-algebra such that $\mathbf{M}(\texttt{i})$ 
is equivalent to $B$-proj. Moreover, let $1 = \epsilon_1 + \epsilon_2 + \cdots + \epsilon_r$ 
be a decomposition of the identity in $B$ into a sum of pairwise orthogonal 
primitive idempotents. Similarly to the situation in $A$, we denote, 
for $1 \leq i,j \leq r$, by $G_{ij}$ the endofunctor of $B$-mod 
given by tensoring with the indecomposable projective $B$-$B$-bimodule $B\epsilon_i \otimes \epsilon_jB$. 
Note that, a priori, there is no reason why we should have $r = n$.

For $i=1,2,\dots,r$, we denote by $Q_i$ the projective $B$-module $B\epsilon_i$.

We may, without loss of generality, assume that $\mathbf{M}$ 
is faithful since $\mathscr{C}_A$ is simple which was shown in \cite[Subsection 3.2]{MMZ}.
Indeed, if we assume that $\mathbf{M}$ is not faithful, then $M(F_{ij}) = 0$,
for all $i,j$. However, then the quotient of $\mathscr{C}_A$ by the ideal
generated by all $F_{ij}$ satisfies all the assumptions of 
\cite[Theorem 18]{MM5} and therefore $\mathbf{M}$ is equivalent to the cell $2$-representation
$\mathbf{C}_{\mathbbm{1}_{\texttt{i}}}$ in this case. 

So let us from now on assume that $\mathbf{M}$ is faithful and, in particular, that all $\mathbf{M}(F_{ij})$ 
are non-zero. As we have seen above,  $A$ has a non-zero projective-injective module and thus, 
combining \cite[Section~3]{MZ1} and \cite[Theorem~2]{KMMZ}, we deduce that 
each $\mathbf{M}(F_{ij})$ is a projective endofunctor of 
$B$-mod and, as such, is isomorphic to a non-empty direct sum of $G_{st}$, for some $1 \leq s,t \leq r$, 
possibly with multiplicities.

\subsection{The sets $X_i$ and $Y_i$}
Following \cite{MZ2}, for $1\leq i,j\leq n$, we define 
\begin{itemize}
 \item $ X_{ij} := \{s\;|\;$ $G_{st}$ is isomorphic to a direct summand of $\mathbf{M}(F_{ij})$, 
 for some $1 \leq t \leq r\}$,
 \item $ Y_{ij} := \{t\;|\;$ $G_{st}$ is isomorphic to a direct summand of $\mathbf{M}(F_{ij})$, 
 for some $1 \leq s \leq r\}$.
\end{itemize}

First of all, note that $X_{ij}$ and $Y_{ij}$ are non-empty as each $\mathbf{M}(F_{ij})$ is non-zero 
due to faithfulness of $\mathbf{M}$.

In \cite[Lemma~20]{MZ2}, it is shown that $X_{ij_1} = X_{ij_2}$, 
for all $j_1,j_2 \in \{1,\dots, n\}$, and thus we may denote by 
$X_i$ the common value of all $X_{ij}$. Similarly, 
the sets $Y_{ij}$ only depend on $j$, hence we may denote by $Y_j$ the common value of the 
$Y_{ij}$, for all $i$. In \cite[Lemma~22, Lemma~21]{MZ2}, it is shown that $X_q = Y_q$, for all $q$ 
and moreover that $X_1\cup X_2\cup\dots\cup X_n=\{1,2,\dots,r\}$.
\subsection{Analysis of the sets $X_i$}
\label{Setup}

For a $1$-morphism $H$ in $\mathscr{C}_A$, we will denote by  $[H]$ the $r\times r$ matrix with coefficients
$h_{st}$, where $s,t\in\{1,2,\dots,r\}$, such that $h_{st}$ gives the multiplicity
of $Q_s$ in $HQ_t$.

\begin{lemma}
\label{1element}
 For each $i \in \{1,\ldots, n\}$, we have $|X_i| = 1$,
\end{lemma}

\begin{proof}
First we note that, for all $1\leq i \leq n$, we have 
$F_{ii}\circ  F_{ii}\cong F_{ii}^{\oplus \dim(e_iAe_i)}$
and hence  $[F_{ii}]^2 = k_i[F_{ii}]$, where 
\begin{displaymath}
k_i=
\begin{cases}
2, & i\not\in S;\\
1, & i\in S.
\end{cases}
\end{displaymath}
Hence we can apply Theorem \ref{Flor} to $\frac{1}{k_i}[F_{ii}]$.
This yields that there exists an ordering of the basis vectors such that
\begin{equation}
\label{quasi-idempotent}
 \frac{1}{k_i}[F_{ii}] =
   \begin{pmatrix}
    0 & AJ & AJB \\
    0 & J  & AB  \\
    0 & 0  & 0 
   \end{pmatrix}.
\end{equation}
However, we have seen that $X_i = Y_i$, for any ordering of the basis. This implies that the 
if the $l$-th row of $\frac{1}{k_i}[F_{ii}]$ is zero, then so is the $l$-th column.  
Thus we get that $A = B = 0$ and, in particular, that 
\[
 \frac{1}{k_i}[F_{ii}] = 
 \begin{pmatrix}
    0 & 0\\
    0 & J
  \end{pmatrix}.
\]
If $i\in S$, we are done as $k_i=1$ and thus the trace of the corresponding matrix is $1$
which yields that $[F_{ii}]$ has to contain exactly one non-zero diagonal element and thus
$|X_i| = 1$.

Let now $i\not\in S$. We may restrict the action of $\mathscr{C}_A$ 
to the $2$-full finitary $2$-subcategory $\mathscr{D}$ of $\mathscr{C}_A$ whose 
indecomposable $1$-morphisms  are the ones which are isomorphic
to either $\mathbbm{1}_\texttt{i}$ or $F_{ii}$. This $2$-category, 
clearly, has only strongly regular two-sided cells. As $i\not\in S$, the projective module
$P_i$ is also injective and hence $F_{ii}$ is a self-adjoint functor (see \cite[Subsection~7.3]{MM1}).  
Therefore
$\mathscr{D}$ satisfies all assumptions of \cite[Theorem 18]{MM5} and hence every
simple transitive $2$-representation of $\mathscr{D}$ is equivalent to a cell 
$2$-representation. 

The $2$-category $\mathscr{D}$ has two left cells (both are also two-sided cells) and each
left cell contains a unique indecomposable $1$-morphism. The matrix of $F_{ii}$ in these
$2$-representations is either $(0)$ or $(2)$. This implies that, for $i\not\in S$, all diagonal elements in 
$[F_{ii}]$ are either equal to $0$ or to $2$. As $[F_{ii}]$ has trace $2$, it follows again that 
$[F_{ii}]$ contains a unique non-zero diagonal element and thus
$|X_i| = 1$.
\end{proof}

Next we are going to prove that the $X_i$'s are mutually disjoint. 

\begin{lemma}
\label{disjoint}
For $i,j \in \{1,\ldots, n\}$ such that $i \neq j$, we have $X_i \cap X_j = \varnothing$.
\end{lemma}

\begin{proof}
Let $i,j \in \{1,\ldots, n\}$ be such that $i \neq j$ and assume that 
$X_i \cap X_j \neq \varnothing$. This implies, by Lemma \ref{1element}, that 
$X_i = X_j$ is a singleton, call it $X_i = \{s\}$. 
By the above, we have that $Y_i = X_i = X_j = Y_j = \{s\}$. This implies that 
\begin{align*}
\mathbf{M}(F_{ii}) = \mathbf{M}(F_{jj}) =  \mathbf{M}(F_{ij}) = \mathbf{M}(F_{ji}) = G_{ss}.
\end{align*}

We have, for any $1 \leq k,l \leq n$, 
\begin{equation}\label{eqeqeq}
  \dim{e_lAe_k} = \begin{cases}
                   2, \quad &k = l\not\in S,\\
                   1, &\{k,l\}\in E \text{ or } k = l \in S,\\
                   0, &\text{else}.
                  \end{cases}
\end{equation}

On the one hand, we know that $\dim{\epsilon_sB\epsilon_s} \geq 1$ 
and thus $G_{ss}\circ G_{ss} \neq 0$. On the other hand, we have that 
$F_{ij} \circ F_{ij} =F_{ij}^{\oplus \dim(e_jAe_i)}$.
This implies that $\dim(e_jAe_i)\neq 0$ and hence $\{i,j\}\in E$,
because of  \eqref{eqeqeq}. Further, as we assume that $S\neq V$, we also have
$\{i,j\}\not\subset S$. Let us assume that $i\not\in S$.

As $i\not\in S$, \eqref{eqeqeq} yields the following: 
\begin{align*}
G_{ss}^{\oplus 2} & = \mathbf{M}(F_{ii}^{\oplus 2}) = 
\mathbf{M}(F_{ii}\circ F_{ii}) = \mathbf{M}(F_{ii})\mathbf{M}(F_{ii}) \\
&= \mathbf{M}(F_{ji})\mathbf{M}(F_{ji}) = \mathbf{M}(F_{ji}\circ F_{ji}) = \mathbf{M}(F_{ji}) = G_{ss}.
\end{align*}
As $G_{ss}\neq 0$, this equality is impossible. The obtained contradiction proves our claim.
\end{proof}

From the above, we have $n=r$ and, without loss of generality, we may assume $X_i=\{i\}$,
for all $i=1,2,\dots,n$.

\begin{corollary}
\label{equalCartan}
 For $i, j = 1, 2,  \ldots, n$, we have $\dim{e_iAe_j} = \dim{\epsilon_iB\epsilon_j}$.
\end{corollary}

\begin{proof}
This follows immediately since every $F_{st}$ acts via $G_{st}$, by comparing
\begin{align*}
F_{si}\circ F_{jt} \cong F_{st}^{\oplus \dim{e_iAe_j}} \quad \text{ with } \quad G_{si}\circ G_{jt} 
\cong G_{st}^{\oplus \dim{\epsilon_iB\epsilon_j}}
\end{align*}
\end{proof}

\subsection{Proof of Theorem \ref{mainresult}}
With the results of Subsection \ref{Setup} at hand, the proof of Theorem \ref{mainresult} 
can now be done using similar 
arguments as used in \cite[Section 5]{MZ2} or \cite[Subsection 4.9]{MaMa}. 
Consider the principal $2$-representation 
$\mathbf{P}_{\texttt{i}} := \mathscr{C}_A(\texttt{i}, \_ )$ of $\mathscr{C}_A$, 
that is the regular action of $\mathscr{C}_A$ on 
$\mathscr{C}_A(\texttt{i}, \texttt{i})$. Set $\mathbf{N} := \text{add}(F_{i1})$, 
where $i = 1, 2, \ldots, k$, be the additive 
closure of all $F_{i1}$. Now, $\mathbf{N}$ is $\mathscr{C}_A$-stable and thus 
gives rise to a $2$-representation of $\mathscr{C}_A$. 
By \cite[Subsection 6.5]{MM2}, we have that there exists a unique 
$\mathscr{C}_A$-stable left ideal $\mathbf{I}$ in $\mathbf{N}$ and 
the corresponding quotient is exactly the cell $2$-representation $\mathbf{C}_{\mathcal{L}_1}$.

Now, mapping $\mathbbm{1}_{\texttt{i}}$ to the simple object corresponding to $Q_1$
in the abelianization of $\mathbf{M}$, induces a  
$2$-natural transformation $\Phi: \mathbf{N} \to \mathbf{M}$.
Due to the results of the previous subsection, we know that $\Phi$ maps indecomposable 
$1$-morphisms in $\mathcal{L}_1$ to indecomposable objects in $\mathbf{M}$ inducing
a bijection on the corresponding isomorphism classes. By uniqueness of the maximal ideal, 
the kernel of $\Phi$ is contained in $\mathbf{I}$. However, by 
Corollary \ref{equalCartan}, the Cartan matrices of $A$ and $B$ are the same. 
This implies that, on the one hand, the kernel of $\Phi$  cannot be smaller than $\mathbf{I}$
and, on the other hand, that $\Phi$ must be full. Therefore $\Phi$ induces an 
equivalence between $\mathbf{N}\slash\mathbf{I}\cong\mathbf{C}_{\mathcal{L}_1}$ and $\mathbf{M}$. 
The claim of the theorem follows.

{\bf Acknowledgments.}  The author wants to thank his supervisor Volodymyr Mazorchuk for many helpful discussions. 

\bibliographystyle{alpha}

\end{document}